\documentclass[12pt]{amsart}

\usepackage{amssymb}
\usepackage{bm}
\usepackage{graphicx}
\usepackage{enumerate}
\usepackage{color}
\usepackage{url}

\newcommand{\excise}[1]{}

\newtheorem{thm}{Theorem}

\newtheorem{question}[thm]{Question}

\theoremstyle{definition}

\newtheorem{remark}[thm]{Remark}

\newtheorem{ans}[thm]{Answer}

\newcommand{\ring}[1]{\ensuremath{\mathbb{#1}}}

\newcommand\NN{\ring{N}}

\newcommand\ZZ{\ring{Z}}
\newcommand\bb{{\mathbf b}}
\newcommand\kk{\Bbbk}

\newcommand\xx{{\mathbf x}}

\newcommand\ttt{\mathbf{t}}
\renewcommand\aa{{\mathbf a}}

\begin{document}

\title{Finding all monomials in a polynomial ideal}
\author{Ezra Miller}
\address{Mathematics Department\\Duke University\\Durham, NC 27708}
\urladdr{\url{http://math.duke.edu/people/ezra-miller}}

\makeatletter
  \@namedef{subjclassname@2010}{\textup{2010} Mathematics Subject Classification}
\makeatother

\date{31 May 2016}

\begin{abstract}
Given a $d \times n$ integer matrix~$A$, the main result is an
elementary, simple-to-state algorithm that finds the largest
$A$-graded ideal contained in any ideal~$I$ in a polynomial ring
$\kk[\xx]$ in~$n$ variables.  The special case where $A$ is an
identity matrix yields that $(\ttt.I) \cap \kk[\xx]$ is the largest
monomial ideal in~$I$, where the generators of $\ttt.I$ are those
of~$I$ but with each variable $x_i$ replaced by $t_ix_i$ for an
invertible variable~$t_i$.
\end{abstract}
\maketitle

\noindent
It is easy to tell whether an ideal $I$ in a polynomial ring $\kk[\xx]
= \kk[x_1,\ldots,x_n]$ contains at least one monomial: it does so if
and only if the saturation $(I:(x_1\cdots x_n)^\infty)$ is the unit
ideal.  Being more precise about the monomials in~$I$ makes the
problem a little harder.  Here are three equivalent ways to formulate
it.

\begin{question}\label{q:monomials}
Fix an ideal $I \subseteq \kk[\xx]$.
\begin{enumerate}
\item%
What is the set of monomials in~$I$?
\item%
What is the largest $\NN^n$-graded ideal contained in~$I$?
\item%
What is the smallest $(\kk^*)^n$-scheme containing the zero scheme
of~$I$?
\end{enumerate}
\end{question}

\begin{ans}\label{a:monomial}
Let $\ttt = t_1,\ldots,t_n$ be a new set of variables.  Inside of the
Laurent polynomial ring $\kk[\xx][\ttt^{\pm1}]$, let $\ttt.I$ denote
the ideal whose generators are those of~$I$ where each variable $x_i$
is replaced by $t_i x_i$.  The biggest monomial ideal contained in~$I$
is $(\ttt.I) \cap \kk[\xx]$.
\end{ans}

This answer appears, with non-invertible $\ttt$-variables, as
Algorithm~4.4.2 in \cite{sst00}.  An elementary proof is given there.
It is obvious, for instance, that every monomial in~$I$ lies in
$(\ttt.I) \cap \kk[\xx]$, since the $\ttt$ variables are units; and
intuitively, there is no way to clear all of the $\ttt$ variables
simultaneously from all of the monomials in a given polynomial with
more than one term.  That said, viewing Question~\ref{q:monomials}
as a special case of a more general problem from multigraded algebra
lends insight.  For notation, if $A \in \ZZ^{d \times n}$ is a $d
\times n$ matrix of integers, to say that the polynomial ring
$\kk[\xx]$ is \emph{$A$-graded} means that each monomial $\xx^\bb \in
\kk[\xx]$ is assigned the \emph{$A$-degree} $\deg(\xx^\bb) = A\bb$,
the linear combination of the $n$ columns of the matrix~$A$ with
coefficients $\bb = b_1,\ldots,b_n$.  An ideal $I$ is
\emph{$A$-graded} if it is generated by polynomials whose terms all
have the same~$A$-degree.

\begin{thm}\label{t:dxn}
Fix an ideal $I \subseteq \kk[\xx]$ and a $d \times n$ matrix~$A$ with
column vectors\/ $\aa_1,\ldots,\aa_n$.  Let $\ttt = t_1,\ldots,t_n$ be
a new set of variables.  Denote by $\ttt.I \subseteq
\kk[\xx][\ttt^{\pm1}]$ the ideal whose generators are those of~$I$
with each variable $x_i$ replaced by $\ttt^{\aa_i} x_i$.  The largest
$A$-graded ideal contained in~$I$ equals the intersection $(\ttt.I)
\cap \kk[\xx]$.
\end{thm}

After the first version of this note was posted, the authors of
\cite{kreuzer-robbiano05} pointed out that the statement of
Theorem~\ref{t:dxn} is essentially Tutorial~50(a) in their book, an
exercise with a suggested proof that is different from the one here.

\begin{remark}\label{r:A-graded}
Details on $A$-graded algebra in general can be found in
\cite[Chapter~8]{cca}.  The $A$-grading on $\kk[\xx]$ corresponds
uniquely to the action of a torus $T \cong (\kk^*)^d$ on~$\kk^n$.
(References for this are hard to locate.  An exposition appears in
Appen\-dix~A.1 of the first arXiv version of~\cite{grobGeom}, at
{\small\texttt{http://arxiv.org/abs/math/0110058v1}}.)  Under this
correspondence, $A$-graded ideals correspond to subschemes of~$\kk^n$
that carry $T$-actions.  Therefore the zero scheme of the ideal
$(\ttt.I) \cap \kk[\xx]$ in Theorem~\ref{t:dxn} is the smallest
$T$-scheme containing the zero scheme $Z(I)$.
\end{remark}

\begin{proof}[Proof of Theorem~\ref{t:dxn}]
Let $X = T \times \kk^n$.  Create a subbundle $Y \subseteq X$ over~$T$
whose fiber over $\tau \in T$ is the translate $\tau^{-1}.Z(I)$ of the
zero-scheme~$Z(I)$ by~$\tau^{-1}$.
The image of the projection of~$Y$ to~$\kk^n$ is the minimal
$T$-stable scheme containing~$Z(I)$ by construction: it is the union
of all $T$-translates of~$Z(I)$.  Therefore the vanishing ideal of the
image of the projection is the maximal $A$-graded subideal of~$I$.
The scheme $Y$ is expressed, in coordinates, as the zero scheme
of~$\ttt.I$, and the image of its projection to~$\kk^n$ is the zero
scheme of $(\ttt.I) \cap \kk[\xx]$.
\end{proof}

\begin{remark}\label{r:binomial}
In contrast to the monomial situation, the binomial analogue of
Question~\ref{q:monomials}.1, which begins with, ``Is there a binomial
in~$I$?'', appears to be much harder than the monomial question, as
observed by Jensen, Kahle, and Katth\"an \cite{jensen-kahle-katthan}.  They
note, for example, that
for each $d$ there is an ideal in $\kk[x,y]$ that contains no
binomials of degree less than~$d$ but nonetheless has a quadratic
Gr\"obner basis and contains a binomial of degree~$d$.
\end{remark}


\end{document}